\documentclass[12pt]{amsart}

\usepackage[margin=1in]{geometry}
\usepackage{setspace}
\usepackage{enumitem}
\usepackage{tikz}

\numberwithin{equation}{section}

\newtheorem{lem}{Lemma}[subsection]
\newtheorem{lemm}{Lemma}[section]
\newtheorem{cor}{Corollary}[section]
\newtheorem{thm}{Theorem}[section]
\newtheorem*{OZ}{Lemma}
\theoremstyle{definition}\newtheorem{defn}{Definition}[section]

\DeclareMathOperator{\C}{C*}
\DeclareMathOperator{\M}{M}

\begin{document}

\title{Villadsen algebras are singly generated}

\author{Chun~Guang Li}
\address{School of Mathematics and Statistics\\
Northeast Normal University\\
Changchun 130024\\
P. R. China}
\email{licg864@nenu.edu.cn}
\thanks{The result in this paper was obtained during the visit of the first named author to the University of Wyoming in 2023--2024. The first named author thanks the Department of Mathematics and Statistics at the University of Wyoming for its hospitality and the AMS China Exchange Program Ky and Yu-Fen Fan Fund Travel Grants for its support.}

\author{Zhuang Niu}
\address{Department of Mathematics and Statistics\\
University of Wyoming\\
Laramie, WY 82071\\
USA}
\email{zniu@uwyo.edu}
\thanks{The research of the second named author is supported by a Simons Foundation grant (MPTSM-00002606).}

\author{Vincent~M. Ruzicka}
\address{Department of Mathematics and Statistics\\
University of Wyoming\\
Laramie, WY 82071\\
USA}
\email{vruzicka@uwyo.edu}


\begin{abstract}
	We show that Villadsen algebras, which are not $\mathcal Z$-stable, are singly generated. More generally, we show that any simple unital AH algebra with diagonal maps is singly generated. 
\end{abstract}
\maketitle

\section{Introduction}
	The generator problem asks about the minimal number of generators for a given C*-algebra---see \cite{TW} for a survey of this problem. In particular, one wonders if a given C*-algebra is \emph{singly} generated, and it is an interesting open question whether or not every simple separable unital C*-algebra is singly generated. In \cite{TW}, it is shown that every simple separable unital \emph{$\mathcal Z$-stable} C*-algebra $B$ is singly generated (i.e., $B \otimes \mathcal Z \cong B$, where $\mathcal Z$ denotes the Jiang-Su algebra \cite{JS}). But Villadsen algebras (of the first type, Definition \ref{D:Vill} below) provide examples of simple unital C*-algebras which are \emph{not} $\mathcal Z$-stable, and the motivation for the current work is to show (Corollary \ref{C:main2}) that these algebras are nevertheless singly generated. 
	
	Being non-$\mathcal Z$-stable, Villadsen algebras are not covered by the current classification theorem for C*-algebras (\cite{Eetal17}, \cite{Eetal}, \cite{EN}, \cite{TWW}, \cite{GLN1}, \cite{Cetal}, \cite{GLN2}). However, one regularity property they do possess is stable rank one; that is, the invertible elements in a Villadsen algebra form a norm dense subset of the algebra. Moreover, some partial classification results for Villadsen algebras are obtained in \cite{ELN} using the radius of comparison (or Cuntz semigroup).

	In this paper, we show that Villadsen algebras are singly generated. In fact, we show that any simple unital AH algebra with diagonal maps (Definition \ref{D:AH} below) is singly generated. To show these algebras are singly generated, we introduce the following concept: a C*-algebra $B$ has an AF-action if it contains a simple AF algebra $A$ and a C*-subalgebra $D$ such that 
\begin{enumerate}
\item $B$ is generated by $A$ and $D$,
\item $D$ commutes with a certain ``diagonal'' subalgebra of $A$,
\item $vdv^* \in D$ for any $d \in D$ and $v \in V$, 
\end{enumerate}
where $V$ denotes a special set of partial isometries in $A$ which are intimately connected with the diagonal subalgebra in the second condition (for a precise statement, see Definition \ref{D:AF}). Our main theorem (Theorem \ref{T:main}) states that any C*-algebra with an AF-action is singly generated. 

	In Section \ref{S:prelim}, we establish some definitions and some simple consequences of these definitions, and the remainder of the paper is dedicated to proving our main theorem at the end of Section \ref{S:main}. 


\section{Definitions}\label{S:prelim}

\begin{defn}\label{D:Vill}
	Let $(c_i)_{i\in \mathbb N}$, $(k_i)_{i\in \mathbb N}$ and $(l_i)_{i\in \mathbb N}$ be sequences of natural numbers, $X$ be a compact connected metric space, and $E_i$ be a set of cardinality $k_i$ for each $i \in \mathbb N$ such that, writing $X_1 = X$ and $X_{i+1} = X_{i}^{c_i}$,
\begin{enumerate}
	\item $E_i \subseteq X_i$, \label{Con:Vill1}
	\item the set
\begin{align*}
	E_{i+1} \cup \bigg ( \bigcup_{s=1}^{c_{i+1}} \pi_s(E_{i+2}) \bigg ) \cup \bigg (\bigcup_{j=3}^\infty \bigcup_{s=1}^{c_{i+1}\cdots c_{i+j-1}} \pi_s(E_{i+j})\bigg ) 
\end{align*}
is dense in $X_{i+1}$, where $\pi_s$ denotes the coordinate projection, \label{Con:Vill2}
	\item $\lim\limits_{i\to\infty}  \frac{l_1\cdots l_{i}}{(l_1+k_1)\cdots(l_{i}+ k_{i})} \not = 0$, \label{Con:Vill3}
\end{enumerate}
for each $i \in \mathbb N$. A \textit{Villadsen algebra} is the limit of an inductive sequence $(B_i,\phi_i)_{i\in \mathbb N}$ of C*-algebras, where $B_i = \M_{n_{i-1}}(\text C(X_i))$, $n_0 \in \mathbb N$ and $n_i = n_{i-1}(l_{i}+k_{i})$, and the seed for $\phi_i$ is given by 
\begin{multline*}
	\text C(X_i) \ni f \mapsto \text{diag}\Big \{\underbrace{f\circ \pi_1,\dotsc,f\circ \pi_1}_{s_{i,1}},\dotsc, \underbrace{f\circ \pi_{c_i},\dotsc,f\circ \pi_{c_i}}_{s_{i,c_i}},f(x_{i,1}),\dotsc,f(x_{i,k_i})\Big \}\\ \in \M_{l_{i}+k_i}\big ( \text C(X_{i+1})\big ), 
\end{multline*}
where $s_{i,t} \in \mathbb N$ for each $1 \leq t \leq c_i$ and $s_{i,1}+\cdots+s_{i,c_i} = l_i$. 
\end{defn}

Note that the above definition of a Villadsen algebra is more general than the original construction in \cite{V}; in addition to the algebras in \cite{V}, Definition \ref{D:Vill} also includes as a special case some algebras constructed by Goodearl in \cite{G} (see \cite[Remark 2.1]{ELN}). 

Given an arbitrary inductive sequence $(B_i,\phi_i)_{i\in \mathbb N}$ of C*-algebras, define $\phi_{i,i'} := \phi_{i'-1} \circ \cdots \circ \phi_{i}$ for $i' > i +1$ and $\phi_{i,i+1}: = \phi_i$. Suppose $(B_i,\phi_i)_{i\in \mathbb N}$ is as in Definition \ref{D:Vill}. Then $\phi_i$ is unital, and a direct calculation shows that for $i' > i+ 1$ the seed for $\phi_{i,i'}$ is (up to permutation) given by 
\begin{align}\label{E:composition}
	\text C(X_i) \ni f \mapsto \text{diag}\Big \{ &\underbrace{f\circ \pi_1,\dotsc,f\circ \pi_{c_i\cdots c_{i'-1}}}_{l_i\cdots l_{i'-1}},\\ \notag &\underbrace{f\circ \pi_1 (E_{i'-1}),\dotsc,f\circ \pi_{c_i\cdots c_{i'-2}} (E_{i'-1})}_{l_i\cdots l_{i'-2}},\\
\notag &\underbrace{f\circ \pi_1(E_{i'-2}),\dotsc,f\circ \pi_{c_i\cdots c_{i'-3}}(E_{i'-2})}_{l_{i}\cdots l_{i'-3}}1_{l_{i'-1}+k_{i'-1}},\\ \notag &\underbrace{f\circ \pi_1(E_{i'-3}),\dotsc,f\circ \pi_{c_{i}\cdots c_{i'-4}}(E_{i'-3})}_{l_i\cdots l_{i'-4}}1_{(l_{i'-1}+k_{i'-1})(l_{i'-2}+k_{i'-2})},\dotsc,\\ \notag &\underbrace{f\circ \pi_1(E_{i+1}),\dotsc,f\circ \pi_{c_i}(E_{i+1})}_{l_i}1_{(l_{i'-1}+k_{i'-1})\cdots(l_{i+2}+k_{i+2})},\\ \notag &f(E_i)1_{(l_{i'-1}+k_{i'-1})\cdots(l_{i+1}+k_{i+1})}\Big \} \in \M_{(l_i+k_i)\cdots(l_{i'-1}+k_{i'-1})}\big ( \text C(X_{i'}) \big ). 
\end{align}

\begin{defn}\label{D:AH}
	Let $B$ be the limit of an inductive sequence $(B_i,\phi_i)_{i\in \mathbb N}$ of unital C*-algebras, where each $\phi_i$ is unital. We call $B$ a (unital) \textit{AH algebra with diagonal maps} if $B_i = \bigoplus_{1 \leq j \leq K_i} \M_{n_{i,j}}(\text C(X_{i,j}))$, where $X_{i,j}$ is a compact connected metric space and $n_{i,j},K_i \in \mathbb N$, and if for any $i' > i$ the restriction of the map $\phi_{i,i'}$ to any direct summands $\M_{n_{i,j}}(\text C(X_{i,j}))$ and $\M_{n_{i',j'}}(\text C(X_{i',j'}))$ has a seed of the form $f \mapsto 0$ or $f \mapsto \text{diag}\{f\circ \lambda_1,\dotsc, f\circ \lambda_{m}\}$ for some continuous maps $\lambda_1,\dotsc,\lambda_m \colon X_{i',j'} \to X_{i,j}$. 
\end{defn}

	If $(B_i,\phi_i)_{i\in\mathbb N}$ is as in Definition \ref{D:AH} and each $\phi_i$ is injective, it is well-known that the limit algebra $B$ is simple if and only if, for any $i \in \mathbb N$ and nonzero $b \in B_i$, there is an $i_0 \geq i$ such that for every $i' \geq i_0$, $\phi_{i,i'}(b)(x) \not = 0$ for every $x \in \bigsqcup_{1\leq j \leq K_{i'}}X_{i',j}$ (\cite[Proposition 2.3]{EHT}). From this characterization of simplicity for $B$, one sees that the unital AF subalgebra $A$ of $B$ obtained as the limit of the inductive sequence $(A_i,\psi_i)_{i\in \mathbb N}$, where $A_i = \bigoplus_{1\leq j \leq K_i} \M_{n_{i,j}}$ and $\psi_i = \phi_i|_{A_i}$, is simple if $B$ is. 
	
	Letting $(B_i,\phi_i)_{i\in\mathbb N}$ be as in Definition \ref{D:Vill} once again, it is clear that the Villadsen algebra arising from this inductive sequence is an AH algebra with diagonal maps with injective connecting maps; moreover, from the above characterization of simplicity for such an algebra, one sees almost immediately that it is simple. Indeed, fix $i \in \mathbb N$ and let $f \in B_{i+1}$ be nonzero; then, by Condition \ref{Con:Vill2} of Definition \ref{D:Vill}, there exists some $i_0 \geq i+1$ and $y \in E_{i_0-1}$ for which $f\circ \pi_s (y) \not = 0$; we then see from Equation \eqref{E:composition} that for every $i' \geq i_0$, $\phi_{i+1,i'}(f)(x) \not = 0$ for any $x \in X_{i'}$.

	
\begin{defn}\label{D:AF}
	Let $B$ be a unital C*-algebra containing a simple separable unital AF subalgebra $A$ and a separable unital C*-subalgebra $D$. Let $(\bigoplus_{1\leq j\leq K_i} \M_{n_{i,j}},\phi_i)_{i\in \mathbb N}$ be a canonical inductive limit decomposition for $A$, where $n_{i,j},K_i \in \mathbb N$. Denote the set of canonical matrix units for $\M_{n_{i,j}}$ by $V_{i,j}$, and define 
\begin{align*}
	V:=\bigcup_{i\in \mathbb N} \bigcup_{j=1}^{K_i} V_{i,j}, \quad E_{i,j} := \{v\in V_{i,j} \mid v=v^*\}, \quad D_0 := \C\Bigg(\bigcup_{i\in \mathbb N} \bigcup_{j=1}^{K_i} E_{i,j}\Bigg). 
\end{align*}
We say that $B$ has an \textit{AF-action} if 
\begin{enumerate}[ref=\textup{(\arabic*)}]
	\item $B = \C(A,D)$, \label{Con:AF1}
	\item $[d,d'] = 0$ for any $d\in D$ and $d' \in D_0$,  \label{Con:AF2}
	\item $vdv^* \in D$ for any $d \in D$ and $v \in V$. \label{Con:AF3}
\end{enumerate} 
In the sequel, to emphasize the dependence of $B$ on $A$ and $D$, we may write $B(A,D)$ for $B$; moreover, associated to $B$ is the inductive limit decomposition $(\bigoplus_{1\leq j\leq K_i} \M_{n_{i,j}},\phi_i)_{i\in \mathbb N}$ of $A$, which we may simply refer to as the ``associated decomposition of $A$.''
\end{defn}

	The following lemmas are straightforward consequences of this definition. 
	
\begin{lemm}
	A simple AH algebra with diagonal maps has an AF-action. In particular, a Villadsen algebra has an AF-action. 
\end{lemm}

\begin{proof}
	Let $B = \lim_{i\to\infty}(B_i,\phi_i)$ be a simple AH algebra with diagonal maps, where $B_i$ and $\phi_i$ are as in Definition \ref{D:AH}, and let $A = \lim_{i\to\infty}(A_i,\psi_i)$ be the AF subalgebra of $B$ as described in the paragraph following Definition \ref{D:AH}. Denote the set of canonical matrix units for $\M_{n_{i,j}}$ by $V_{i,j}$ and define the sets $V$, $E_{i,j}$, and $D_0$ as in Definition \ref{D:AF}. Moreover, define 
\begin{align*}
	D := \C \Big ( \{p\otimes f \mid p \in E_{i,j},\, f \in \text C(X_{i,j}),\, i \in \mathbb N, \, 1 \leq j \leq K_i\} \Big ).
\end{align*}

	Notice $D$ is a separable unital C*-subalgebra of $B$. Furthermore, that $B = \C(A,D)$ follows from the fact that any $b \in\M_{n_{i,j}}(\text C(X_{i,j}))$ may be written as a finite sum of elements of the form $v (p\otimes f) v'$ for $v,v' \in V_{i,j}$, $p \in E_{i,j}$, and $f \in \text C(X_{i,j})$; since $D$ is commutative and $D_0 \subseteq D$, we have $[d,d'] = 0$ for any $d \in D$ and $d' \in D_0$; finally, clearly $v(p\otimes f)v^* \in D$ for any $v \in V_{i,j}$, $p \in E_{i',j'}$, and $f \in \text C(X_{i',j'})$ so that $v d v^* \in D$ for any $v \in V$ and $d \in D$. 
\end{proof}

\begin{lemm}
	Let $B = B(A,D)$ have an AF-action, and let $C$ be a separable unital C*-algebra. Then $B \otimes C$ has an AF-action. 
\end{lemm}

\begin{proof}	 
	 Define $\mathcal A_i := \bigoplus_{1\leq j\leq K_i} \M_{n_{i,j}}\otimes \mathbb C 1_C$ for each $i \in \mathbb N$, where $(\bigoplus_{1\leq j\leq K_i} \M_{n_{i,j}},\phi_i)_{i\in \mathbb N}$ is the associated decomposition of $A$, so that $\mathcal A= \overline{\bigcup_{i\in\mathbb N} \mathcal A_i}$ is a simple AF subalgebra of $B\otimes C$. Identify the set  $V_{i,j}$ of canonical matrix units for $\M_{n_{i,j}}$ with the set $\mathcal V_{i,j} = \{v\otimes 1_C \mid v \in V_{i,j}\}$ in $B \otimes C$. Then, writing $\mathcal V = \bigcup_{i\in\mathbb N} \bigcup_{1\leq j \leq K_i} \mathcal V_{i,j}$, $\mathcal E_{i,j} = \{\mathbf v \in \mathcal V_{i,j}\mid \mathbf v = \mathbf v^*\}$, $\mathcal D_0 = \C (\bigcup_{i\in\mathbb N} \bigcup_{1\leq j \leq K_i} \mathcal E_{i,j} )$, and $\mathcal D = D \otimes C$, it follows that $B\otimes C = \C(\mathcal A,\mathcal D)$, $[\mathbf d,\mathbf d']=0$ for any $\mathbf d \in \mathcal D$ and $\mathbf d' \in \mathcal D_0$, and $\mathbf v \mathbf d \mathbf v^* \in \mathcal D$ for any $\mathbf d \in \mathcal D$ and $\mathbf v \in \mathcal V$. 
\end{proof}


%
%

\section{A Generator for an Algebra with an AF-action}\label{S:main}

\subsection{Preliminary Lemmas}

	Let $B$ be a unital C*-algebra, and let $a \in B$ be such that $q_iaq_j = 0$ for $i > j$, where $(q_i)_{1 \leq i \leq n} \subseteq B$ is a finite sequence of nonzero mutually orthogonal projections summing to the identity ($a$ is ``upper triangular'' with respect to $(q_i)$). It is well-known that $\sigma(a) \subseteq \bigcup_{1 \leq i \leq n} \sigma(q_iaq_i)$, and the following lemma gives a similar result for infinite sequences. 

\begin{lem}\label{L:upT}
	Let $B$ be a unital C*-algebra, let $a \in B$, and let $(p_i)_{i\in\mathbb N} \subseteq B$ be a sequence of nonzero mutually orthogonal projections such that 
	\begin{enumerate}
		\item $(1-\sum\limits_{i=1}^n p_i)a\sum\limits_{i=1}^n p_i = 0$ for each $n\in \mathbb N$, \label{Con:upT}
		\item $\lim\limits_{n\to\infty} \|(1-\sum\limits_{i=1}^n p_i)a(1-\sum\limits_{i=1}^n p_i)\| = 0$, \label{Con:trace}
		\item $\sigma(p_iap_i) \cap \sigma(p_{i'}ap_{i'}) = \text{\O}$ for $i \not = i'$, \label{Con:upTspec}
	\item $0 \not \in \sigma(p_iap_i)$ for any $i\in \mathbb N$. \label{Con:upTspec2}
	\end{enumerate}
Then 
\begin{align*}
	\sigma(a) \subseteq \bigcup_{i=1}^\infty \sigma(p_iap_i) \cup \{0\}\quad \text{and} \quad (p_i)_{i\in \mathbb N} \subseteq \C(a). 
\end{align*}
\end{lem}

\begin{proof}
	Define $P_n:=\sum_{1\leq i \leq n} p_i$, and write $1-P_n = P_n^\perp$ for each $n \in \mathbb N$; also, define $P_0 := 0$ so that $P_0^\perp = 1$. Notice for any $n \in \mathbb N$, $P_{n-1}^\perp = p_n + P_n^\perp$, $p_nP_n^\perp = 0$, and $P_n^\perp(P_{n-1}^\perp a P_{n-1}^\perp)  p_n = P_n^\perp a P_n p_n = 0$ by Condition \ref{Con:upT}. Fixing $n \in \mathbb N$, it then follows from the paragraph preceding this lemma that $\sigma(P_{n-1}^\perp a P_{n-1}^\perp) \subseteq \sigma(p_nap_n) \cup \sigma(P_n^\perp a P_n^\perp)$; by induction, 
\begin{align}\label{E:upT}
	\sigma(P_{n-1}^\perp aP_{n-1}^\perp ) \subseteq \bigcup_{j=n}^m \sigma(p_{j}ap_j) \cup \sigma(P_{m}^\perp a P_m^\perp), \quad \forall m \geq n. 
\end{align}
If $\lambda \in \sigma(P_{n-1}^\perp aP_{n-1}^\perp )$ and $\lambda \not \in \sigma(p_jap_j)$ for any $j \geq n$, Equation \eqref{E:upT} implies $\lambda \in \sigma (P_j^\perp a P_j^\perp)$ for every $j \geq n$; then, by Condition \ref{Con:trace}, $\lambda = 0$ and 
\begin{align}\label{E:inclusion}
	\sigma(P_{n-1}^\perp aP_{n-1}^\perp ) \subseteq \bigcup_{j=n}^\infty \sigma(p_{j}ap_j) \cup \{0\}.  
\end{align}
Taking $n=1$, we have the desired containment for $\sigma(a)$. 

	Let $n \in \mathbb N$ be arbitrary again. By Equation \eqref{E:inclusion} and Conditions \ref{Con:upTspec} and \ref{Con:upTspec2}, $\sigma(p_nap_n) \cap \sigma(P_{n}^\perp a P_n^\perp) = \text{\O}$. It follows that 
\begin{align}\label{E:upTfinal}
	p_n \in \C(P_{n-1}^\perp aP_{n-1}^\perp);
\end{align}
we refer the reader to \cite[Theorem 1]{OZ} and \cite[p.\ 22]{DP} for the details. If $p_1,\dotsc,p_{n-1} \in \C(a)$, expanding  $P_{n-1}^\perp a P_{n-1}^\perp$ in terms of $p_1,\dotsc,p_{n-1}$ reveals that $P_{n-1}^\perp a P_{n-1}^\perp \in \C(a)$. Now, taking $n =1$ in Equation \eqref{E:upTfinal}, we see $p_1 \in \C(a)$; thus $(p_i)_{i\in \mathbb N} \subseteq \C(a)$ by induction. 
\end{proof}

\begin{lem}\label{L:map}
Let $B$ be a unital C*-algebra, and suppose $B$ contains a unital C*-subalgebra $D$ and a finite set $\{v_{k}\}_{1\leq k \leq n}$ of nonzero partial isometries such that 
\begin{enumerate}[label=(\alph*)]
	\item $v_1$ is the range projection of each $v_k$, i.e., $v_1 = v_{k}v_{k}^*$, \label{Con:map1}
	\item the source projections form a partition of unity for $B$, i.e., $(v_{k}^*v_{k})(v_{k'}^*v_{k'}) = 0$ for $k \not = k'$ and $\sum_{1\leq k \leq n} (v_{k}^*v_{k})= 1$,\label{Con:map2}
	\item the source projections commute with $D$, i.e., $[v_{k}^*v_{k},d]=0$ for any  $d \in D$, \label{Con:map3}
	\item $v_kdv_k^* \in D$ for any $d \in D$.  \label{Con:map4}
\end{enumerate}
Then, the map 
\begin{align*}
	\Phi\colon \M_n\big ((v_1v_1^*)  D (v_1 v_1^*)\big ) \to \C(D,v_1,\dotsc,v_n), \quad [b_{i,j}]_{i,j=1}^{n}\mapsto \sum_{i=1}^{n} \sum_{j=1}^{n} v_i^* b_{i,j}v_{j}
\end{align*}
is a $*$-isomorphism. 
\end{lem}

\begin{proof}
	It is clear that $\Phi$ is linear and preserves adjoints. For multiplicativity, notice
\begin{align*}
	\Phi\big ([b_{i,j}]\big) \Phi \big ([c_{i,j}]\big) &= \Big ( \sum_{i=1}^n \sum_{j=1}^n v_{i}^* b_{i,j} v_{j} \Big ) \Big ( \sum_{i=1}^n \sum_{j=1}^n v_{i}^* c_{i,j} v_{j} \Big ) \\ 
	&= \Big ( \sum_{i=1}^n v_i^* b_{i,1} v_1 + \cdots + \sum_{i=1}^n v_i^* b_{i,n} v_n \Big ) \Big ( \sum_{j=1}^n v_1^* c_{1,j} v_{j} + \cdots + \sum_{j=1}^n v_n^* c_{n,j} v_{j} \Big ) \\
	&= \Big ( \sum_{i=1}^n v_i^*b_{i,1}v_1 \Big ) \Big ( \sum_{j=1}^n v_1^*c_{1,j}v_{j} \Big ) + \cdots + \Big ( \sum_{i=1}^n v_i^*b_{i,n}v_n \Big )\Big ( \sum_{j=1}^n v_n^*c_{n,j}v_{j} \Big) \\
	&=  \sum_{i=1}^n \sum_{j=1}^{n} v_i^*  \Big (\sum_{k=1}^{n} b_{i,k}c_{k,j} \Big )v_j \\ 
	&= \Phi \big ([a_{i,j}]\big ), 
\end{align*} 
where $a_{i,j} = \sum_{1\leq k \leq n}b_{i,k}c_{k,j}$ and where the third equality is a result of Condition \ref{Con:map2}. Hence, $\Phi([b_{i,j}]) \Phi ([c_{i,j}]) = \Phi([b_{i,j}][c_{i,j}])$. 

Now, notice $v_1 = v_1v_1^* = v_11v_1^* \in D$ by Condition \ref{Con:map1} and Condition \ref{Con:map4}. Fixing some $1\leq k \leq n$ and defining $[b_{i,j}]$ such that $b_{1,k} = (v_{1}v_{1}^*)v_1(v_{1}v_{1}^*) = v_1$ and $b_{i,j} = 0$ otherwise, it follows from the previous sentence that $[b_{i,j}]\in \M_n((v_1v_1^*) D (v_1 v_1^*) )$. We then see $v_k$ is in the image of $\Phi$ since $\Phi ([b_{i,j}]  ) = v_1^*v_1v_k = v_1v_k = v_kv_k^*v_k = v_k$. Moreover, for any $d \in D$,
\begin{align}\label{E:map}
	d = \Big (\sum_{i=1}^{n} v_i^*v_i\Big ) d \Big (\sum_{i=1}^{n} v_i^*v_i\Big )= \sum_{i=1}^n (v_i^*v_i) d  (v_i^*v_i)
\end{align}
by Condition \ref{Con:map2} and Condition \ref{Con:map3}; by Condition \ref{Con:map1}, 
\begin{align}\label{E:map2}
	\sum_{i=1}^n (v_i^*v_i) d  (v_i^*v_i) = \sum_{i=1}^nv_{i}^*(v_1v_1^*)d_i(v_1v_1^*)v_i, 
\end{align}
where $d_i = v_{i}dv_i^* \in D$ for each $1 \leq i \leq n$. Putting Equation \eqref{E:map} and Equation \eqref{E:map2} together, we see $d$ is in the image of $\Phi$ since $\Phi([c_{i,j}]) = d$ when $c_{k,k} = (v_1v_1^*)d_k(v_1v_1^*)$, for each $1 \leq k \leq n$, and $c_{i,j} = 0$ otherwise. Thus $\Phi$ is onto. 

Finally, notice if $\Phi([b_{i,j}]) = 0$, then 
\begin{align*}
	0= v_k \Phi\big([b_{i,j}]\big) v_{l}^* =  \sum_{i=1}^{n} \sum_{j=1}^{n} v_kv_i^* b_{i,j}v_{j} v_{l}^* = \delta_{k,i}b_{i,j}\delta_{j,l}
\end{align*}
for every $1 \leq k,l \leq n$, where $\delta$ denotes the Kronecker delta function; in particular, $\Phi$ is injective. 
\end{proof}

Lemma \ref{L:OZcor} below references a result from the paper of Olsen and Zame \cite{OZ}; we reproduce a version of it here for the reader's convenience. 

\begin{OZ}[Olsen and Zame]
	Let $A$ be a unital C*-algebra generated by the $k(k+1)/2$ invertible self-adjoint elements $a_1,\dotsc,a_{k(k+1)/2}$ with pairwise disjoint spectra. Then, $\M_k(A)$ is generated by the upper triangular matrix 
\begin{align*}
	\begin{bmatrix}
		a_1&a_2&\cdots&a_k\\
		0&a_{k+1}&\cdots&a_{2k-1}\\
		\vdots&\ddots&\ddots&\vdots\\
		0&\cdots&0&a_{k(k+1)/2}
	\end{bmatrix}. 
\end{align*}
\end{OZ}

\begin{lem}\label{L:OZcor}
	Let $B$, $D$, and $\{v_k\}_{1\leq k \leq n}$ be as in Lemma \ref{L:map}. Let $m$ be a positive integer such that $n > 2m-1$, and let $\{d_1,\dotsc,d_m\} \subseteq D\setminus \{0\}$ be a subset of self-adjoint elements. Then, there exists an invertible element $\mathfrak g \in B$ such that $d_1,\dotsc,d_m \in \C(\mathfrak g)$. 
\end{lem}

\begin{proof}
	As in Equation \eqref{E:map} and Equation \eqref{E:map2}, we can write $d_i \in \{d_1,\dotsc,d_m\}$ as 
\begin{align*}
	d_i = \sum_{j=1}^nv_j^*(v_1v_1^*)d_{i,j}(v_1v_1^*)v_j, \quad d_{i,j} = v_jd_iv_j^* \in  D. 
\end{align*}
Consider the self-adjoint elements $(v_1v_1^*)d_{i,j}(v_1v_1^*) \in (v_1v_1^*) D(v_1v_1^*)$ for $1 \leq i \leq m$ and  $1\leq j \leq n$. Suppose the distinct nonzero such elements constitute a set $S'$, and let $S = S' \cup \{v_1v_1^*\}$. Denoting the cardinality of $S$ by $N$, $\C(S)$ is a unital C*-algebra generated by $N \leq nm+1 \leq n(n+1)/2$ self-adjoint elements; it is then an simple consequence of the continuous function calculus that $\C(S)$ is generated by $n(n+1)/2$ invertible self-adjoint elements with disjoint spectra, say $a_1,\dotsc,a_{n(n+1)/2}$. It follows from Olsen and Zame that $\M_n(\C(S))$ is generated by an element $g$ of the form 
\begin{align*}
	g = \begin{bmatrix}
		a_1&a_2&\cdots&a_n\\
		0&a_{n+1}&\cdots&a_{2n-1}\\
		\vdots&\ddots&\ddots&\vdots\\
		0&\cdots&0&a_{n(n+1)/2}
	\end{bmatrix}; 
\end{align*}
notice $g$ is invertible since its diagonal entries are (see the paragraph preceding Lemma \ref{L:upT}). 

Now, consider the map $\Phi\colon \M_n((v_1v_1^*)  D (v_1 v_1^*)) \to \C( D,v_1,\dotsc,v_n)$ from Lemma \ref{L:map}. Clearly $\Phi(\C(g)) = \C(\Phi(g))$, and $\Phi(g)$ is invertible. Fix $1 \leq i \leq m$; since $\C(g)$ contains the element $[a_{k,l}]$, where $a_{j,j} = (v_1v_1^*)d_{i,j}(v_1v_1^*)$ for $1 \leq j \leq n$ and $a_{k,l} = 0$ otherwise, $\Phi(\C(g))$ contains the element $\Phi([a_{k,l}]) = d_{i}$. Writing $\mathfrak g = \Phi(g)$, the result follows. 
\end{proof}

\subsection{Lemmas Pertaining Specifically to Algebras with AF-actions}

	Let $B = B(A,D)$ have an AF-action, and let $(\bigoplus_{1\leq j\leq K_i} \M_{n_{i,j}},\phi_i)_{i\in \mathbb N}$ be the associated decomposition of $A$. Fix $i_0 \in \mathbb N$, and denote the multiplicity of the embedding of $\M_{n_{i_0,j'}}$ into $\M_{n_{i,j}}$ via $\phi_{i_0,i}$ by $m_{i_0,i;j',j}$, where $i> i_0$, $1\leq j' \leq K_{i_0}$, and $1\leq j \leq K_{i}$. For a positive integer $N$, note that one can always find an $i > i_0$ such that $m_{i_0,i;j',j}>N$ for each $1\leq j' \leq K_{i_0}$ and $1 \leq j \leq K_{i}$; this is a simple consequence of the fact that $A$ is simple. 


\begin{lem}\label{L:preMain}
	Let $B=B(A,D)$ have an AF-action, and let $(\bigoplus_{1\leq j\leq K_i} \M_{n_{i,j}},\phi_i)_{i\in \mathbb N}$ be the associated decomposition of $A$; denote the set of canonical matrix units for $\M_{n_{i,j}}$ by $V_{i,j}$ and the subset of $V_{i,j}$ consisting of all self-adjoint elements by $E_{i,j}$. Let $p \in E_{i',j'}$ for some $i' \in \mathbb N$ and some $1 \leq j' \leq K_{i'}$, and let $\{d_1,\dotsc,d_m\}\subseteq pDp$ be a subset of self-adjoint elements. Then, there exists an invertible element $\mathfrak g \in pBp$ such that $d_1,\dotsc,d_m \in \C(\mathfrak g)$. 
\end{lem}

\begin{proof}
Let $i$ be such that $m_{i',i;j',j}>2m-1$ for each $1 \leq j \leq K_{i}$, and write $m_{i',i;j',j} = M_{j}$ for convenience. For each $1 \leq j \leq K_i$, there is a subset $\{v_{j,k}\}_{1\leq k \leq M_j} \subseteq V_{i,j}$ of cardinality $M_j$ such that $v_{j,1}$ is the range projection of each $v_{j,k}$ and $p  =\sum_{1 \leq j \leq K_i} \sum_{1 \leq k \leq M_j} (v_{j,k}^*v_{j,k})$; writing $p_j = \sum_{1 \leq k \leq M_j} (v_{j,k}^*v_{j,k})$, it follows that the source projections of the members of the set $\{v_{j,k}\}_{1\leq k \leq M_j}$ form a partition of unity for $p_jBp_j$. Moreover, notice $v_{j,k}$ and $v_{j,k}^*$ commute with $p_j$ so that $v_{j,k}p_jdp_jv_{j,k}^* = p_{j}v_{j,k}dv_{j,k}^*p_j \in p_j Dp_j$ for any $1 \leq k \leq M_j$ and $d \in D$; also, the source projections of the members of the set $\{v_{j,k}\}_{1\leq k \leq M_j}$ are contained in $D_0$ so that $(v_{j,k}^*v_{j,k})p_jdp_j = p_jdp_j(v_{j,k}^*v_{j,k})$ for any $d \in D$. We conclude that for each $1 \leq j \leq K_i$, $p_jBp_j$ is a unital C*-algebra containing a unital C*-subalgebra $p_jDp_j$ and a finite set $\{v_{j,k}\}_{1\leq k \leq M_j}$ of nonzero partial isometries such that 
\begin{enumerate}[label=(\alph*)]
	\item $v_{j,1} = v_{j,k}v_{j,k}^*$,
	\item $(v_{j,k}^*v_{j,k})(v_{j,k'}^*v_{j,k'}) = 0$ for $k \not = k'$ and $\sum_{1\leq k \leq M_j} (v_{j,k}^*v_{j,k}) = p_j$,
	\item $[v_{j,k}^*v_{j,k},d'] = 0$ for any $d' \in p_jDp_j$,
	\item $v_{j,k}d'v_{j,k}^* \in p_j D p_j$ for any $d' \in p_jDp_j$. 
\end{enumerate}

For each $1 \leq j \leq K_i$, consider the self-adjoint elements $p_jd_1p_j,\dotsc,p_jd_mp_j \in p_jDp_j$; take the distinct nonzero such elements and form a set $S_j \subseteq p_jDp_j \setminus \{0\}$. Then, $|S_j|$ (the cardinality of $S_j$) is a positive integer such that $M_j > 2m-1 \geq 2|S_j|-1$. By Lemma \ref{L:OZcor}, there exists an invertible element $g_j \in p_jBp_j$ such that $p_jd_1p_j,\dotsc,p_jd_mp_j \in\C(g_j)$ for each $1 \leq j \leq K_i$. 

Assuming $\sigma(g_j) \cap \sigma(g_{j'}) = \text{\O}$ for $j \not = j'$ (which we may by the functional calculus), we claim that the C*-algebra generated by $\mathfrak g = \sum_{1\leq j \leq K_i} g_j \in pBp$ contains $d_1,\dotsc,d_m$. Indeed, $\mathfrak g$ is ``diagonal'' with respect to the sequence $(p_j)_{1\leq j \leq K_i}$, and it is a simple corollary of Lemma \ref{L:upT} that  $(p_j)_{1\leq j \leq K_i}\subseteq \C(\mathfrak g)$; hence $g_j \in \C(\mathfrak g)$, hence $p_jd_1p_j,\dotsc,p_jd_mp_j \in \C(\mathfrak g)$, for each $1 \leq j \leq K_i$. But, notice $d_l = \sum_{1\leq j \leq K_i} p_jd_lp_j$ for each $1 \leq l \leq m$ (since $D$ and $D_0$ commute). The result follows. 
\end{proof}

\begin{lem}\label{L:preMain2}
	Let $B(A,D)$, $(\bigoplus_{1\leq j\leq K_i} \M_{n_{i,j}},\phi_i)_{i\in \mathbb N}$, $V_{i,j}$, and $E_{i,j}$ be as in Lemma \ref{L:preMain}. Fix $i \in \mathbb N$. For each $1 \leq j \leq K_i$, let $p_j \in E_{i,j}$ and $\{v_{j,k}\}_{1\leq k \leq n_{i,j}} \subseteq V_{i,j}$ be a subset of cardinality $n_{i,j}$ such that $v_{j,1}= p_j$ is the range projection of each $v_{j,k}$. Then, given a self-adjoint element $d \in D$, there exists a subset $G = \{g_j\}_{1\leq j \leq K_i} \subseteq B$ such that 
\begin{enumerate}
\item $g_j \in p_jBp_j$,
\item $0 \not \in \sigma(g_j)$,
\item $d \in \C(\{v_{1,k}\}_{1\leq k \leq n_{i,1}},\dotsc,\{v_{K_i,k}\}_{1\leq k \leq n_{i,K_i}},G)$. 
\end{enumerate} 
\end{lem}

\begin{proof}
	Notice the source projections of the elements of the set $\{v_{j,k}\}_{1\leq k \leq n_{i,j}}$ exhaust the elements of $E_{i,j}$ for each $1 \leq j \leq K_i$ so that the projections in the set $\bigcup_{1\leq j \leq K_i} \{v_{j,k}^*v_{j,k}\}_{1\leq k \leq n_{i,j}}$ form a partition of unity for $B$. We can then write $d$ as 
\begin{align}\label{E:d}
	d &= \Big ( \sum_{j=1}^{K_{i}}\sum_{k=1}^{n_{i,j}} (v_{j,k}^*v_{j,k}) \Big ) d \Big ( \sum_{j=1}^{K_{i}}\sum_{k=1}^{n_{i,j}} (v_{j,k}^*v_{j,k}) \Big )\\\notag
	&= \sum_{j=1}^{K_{i}}\sum_{k=1}^{n_{i,j}} (v_{j,k}^*v_{j,k}) d (v_{j,k}^*v_{j,k}) \\\notag
	&= \sum_{j=1}^{K_{i}}\sum_{k=1}^{n_{i,j}} v_{j,k}^* (p_{j}v_{j,k})d(v_{j,k}^*p_{j})v_{j,k} \\\notag
	&= \sum_{j=1}^{K_{i}}\sum_{k=1}^{n_{i,j}} v_{j,k}^* (p_{j}d_{j,k}p_{j})v_{j,k}
\end{align}
where $d_{j,k} = v_{j,k}dv_{j,k}^*$. Consider the subset $\{p_{j}d_{j,1}p_{j},\dotsc,p_{j}d_{j,n_{i,j}}p_{j}\} \subseteq p_{j}Dp_{j}$ of self-adjoint elements for each $1 \leq j \leq K_i$; Lemma \ref{L:preMain} implies there exists an invertible element $g_{j} \in p_j B p_j$ such that $p_{j}d_{j,1}p_{j},\dotsc,p_{j}d_{j,n_{i,j}}p_{j} \in \C(g_j)$. Defining $G:=\{g_1,\dotsc,g_{K_i}\}$, the final assertion follows from the last equality in Equation \eqref{E:d}. 
\end{proof}

We remark that the subsets $\{v_{j,k}\}_{1\leq k \leq n_{i,j}} \subseteq V_{i,j}$ from the previous lemma generate the finite-dimensional C*-algebra $\bigoplus_{1\leq j\leq K_i} \M_{n_{i,j}}$; that is, 
\begin{align*}
	\bigoplus_{j=1}^{K_i} \M_{n_{i,j}} = \C\big(\{v_{1,k}\}_{k=1}^{n_{i,1}},\dotsc, \{v_{K_i,k}\}_{k=1}^{n_{i,K_i}}\big)
\end{align*}
For this, it is sufficient to show $V_{i,j} \subseteq \C(\{v_{j,k}\}_{1\leq k \leq n_{i,j}})$ for each $1\leq j \leq K_i$. Indeed, fix $j$, and let $V_{i,j} = \{e_{s,t}\mid 1 \leq s,t \leq n_{i,j}\}$. Then, $e_{t_k,t_k} = v_{j,k}^*v_{j,k}$ for each $1 \leq k \leq n_{i,j}$, where $1 \leq t_k \leq n_{i,j}$ and $t_k \not = t_{k'}$ for $k \not = k'$; thus $v_{j,k} = e_{s_k,t_k}$ for some $1 \leq s_k \leq n_{i,j}$. But, $v_{j,1}$ is the range projection of each $v_{j,k}$ so that $v_{j,1} = e_{s_k,s_k} = e_{s_1,t_1}$; hence $s_1 = t_1 = s_k$ for each $1 \leq k \leq n_{i,j}$ so that $v_{j,k} = e_{s_1,t_k}$. Then, to obtain $e_{s,t}$ for some $1 \leq s,t \leq n_{i,j}$, take the product $v_{j,k}^*v_{j,k'}$ for $k$ and $k'$ such that $t_k = s$ and $t_{k'} = t$. 

\subsection{Main Results}

	Let $B=B(A,D)$ have an AF-action, and let $(\bigoplus_{1\leq j\leq K_i} \M_{n_{i,j}},\phi_i)_{i\in \mathbb N}$ be the associated decomposition of $A$; denote the set of canonical matrix units for $\M_{n_{i,j}}$ by $V_{i,j}$ and the subset of $V_{i,j}$ consisting of all self-adjoint elements by $E_{i,j}$. To prove our main theorem, that every C*-algebra with an AF-action is singly generated, we will construct a generator for the C*-algebra $B$. We now define the sets of elements from which we will build a generator for $B$. 
	
	Let $(s_i)_{i\in \mathbb N}$ be a sequence of natural numbers such that $n_{s_1,j} > 1$ for each $1 \leq j \leq K_{s_1}$ and $m_{s_i,s_{i+1};j',j}>1$ for each $1 \leq j' \leq K_{s_i}$ and $1 \leq j \leq K_{s_{i+1}}$; for convenience, write 
\begin{align*}
	\mathsf K_i = K_{s_i}, \quad \mathsf N_{i,j'} = n_{s_i,j'}, \quad \mathsf M_{i+1;j',j} = m_{s_i,s_{i+1};j',j}, \qquad 1 \leq j' \leq \mathsf K_{i}, \ 1 \leq j \leq \mathsf K_{i+1}; 
\end{align*}
moreover, define $\mathsf K_0 := 1$ and $\mathsf M_{1;j',j} := \mathsf N_{1,j}$ for each $1 \leq j' \leq \mathsf K_0$ and $1 \leq j \leq \mathsf K_1$. 

	Inductively on $i$, construct sets of projections $Q_{i;j',j} = \{q_{i;j',j,k}\}_{1\leq k \leq \mathsf M_{i;j',j}}$ for each $1 \leq j' \leq \mathsf K_{i-1}$ and $1 \leq j\leq \mathsf K_i$ of cardinality $\mathsf M_{i;j',j}$ such that 
\begin{enumerate}
	\item[(Q1)] $Q_{i;j',j} \subseteq E_{s_{i},j}$ for each $1\leq j' \leq \mathsf K_{i-1}$, 
	\item[(Q2)] $p_{i-1,j'} = \sum_{1\leq j \leq \mathsf K_{i}} \sum_{q \in Q_{i;j',j}} q$ for each $1 \leq j' \leq \mathsf K_{i-1}$, where $p_{0,1}$ is the identity and $p_{i-1,j'} = q_{i-1;\mathsf K_{i-2},j',\mathsf M_{i-1;\mathsf K_{i-2},j'}}$ for $i > 1$.
\end{enumerate}
We take $Q_{1;1,j} = E_{s_1,j}$ for $1 \leq j \leq \mathsf K_1$ to start the induction. To each set of projections $Q_{i;j',j}$, associate a set of partial isometries $W_{i;j',j} = \{w_{i;j',j,k}\}_{1\leq k \leq \mathsf M_{i;j',j}}$ of cardinality $\mathsf M_{i;j',j}$ such that 
\begin{enumerate}
	\item[(W1)] $W_{i;j',j} \subseteq V_{s_i,j}$ for each $1 \leq j' \leq \mathsf K_{i-1}$,
	\item[(W2)] the range projection of each member of $W_{i;j',j}$ is $w_{i;1,j,1}$ for $1 \leq j' \leq \mathsf K_{i-1}$,
	\item[(W3)] the source projection of $w_{i;j',j,k}$ is $q_{i;j',j,k}$. 
\end{enumerate}
For each $i \in \mathbb N$ and $1 \leq j \leq \mathsf K_i$, choose additional sets of partial isometries $U_{i,j} = \{v_{i;j,k}\}_{1\leq k \leq \mathsf N_{i,j}}$ of cardinality $\mathsf N_{i,j}$ such that 
\begin{enumerate}
\item[(U1)] $W_{i;j',j} \subseteq U_{i,j} \subseteq V_{s_i,j}$ for each $1 \leq j' \leq \mathsf K_{i-1}$, 
\item[(U2)] $v_{i;j,1} = w_{i;1,j,1}$ is the range projection of each member of $U_{i,j}$. 
\end{enumerate}

	We then have the following lemma. 

\begin{lem}\label{L:partialIso}
	For each $i \in \mathbb N$ and $1 \leq j \leq \mathsf K_{i+1}$, one has 
	\begin{align*}
		U_{i+1,j} \subseteq \C\big( \bigoplus_{j=1}^{\mathsf K_i} \M_{\mathsf N_{i,j}}, W_{i+1;1,j},\dotsc,W_{i+1;\mathsf K_i,j}\big). 
	\end{align*}
\end{lem}

\begin{proof}
	Let $V_{s_i,j'} = \{e_{j';l,k}\mid 1 \leq l,k \leq \mathsf N_{i,j'}\}$ for $1 \leq j' \leq \mathsf K_i$ and $V_{s_{i+1},j} = \{f_{j;l,k}\mid 1 \leq l,k\leq \mathsf N_{i+1,j}\}$ for $1\leq j \leq \mathsf K_{i+1}$. Then for each $j$ there exist positive integers $L_{j',j,k,t}$ for $1 \leq j' \leq \mathsf K_i$, $1 \leq k \leq \mathsf N_{i,j'}$, and $1 \leq t \leq \mathsf M_{i+1;j',j}$ such that 
\begin{enumerate}[label=(\alph*)]
	\item $1 \leq L_{j',j,k,t} \leq \mathsf N_{i+1,j}$,\label{prop1}
	\item $L_{j_1',j,k_1,t_1} \not = L_{j_2',j,k_2,t_2}$ for $j_1'\not = j_2'$ or $k_1\not = k_2$ or $t_1\not = t_2$,\label{prop2}
	\item $e_{j';l,k} = \sum\limits_{j=1}^{\mathsf K_{i+1}} \sum\limits_{t=1}^{\mathsf M_{i+1;j',j}} \label{prop3} f_{j;L_{j',j,l,t},L_{j',j,k,t}}$. 
\end{enumerate}
Thus, since for some $1 \leq J_{j'} \leq \mathsf N_{i,j'}$ we have $q_{i;\mathsf K_{i-1},j',\mathsf M_{i;\mathsf K_{i-1},j'}} = e_{j';J_{j'},J_{j'}}$, it follows from Condition \ref{prop3} that $Q_{i+1;j',j} =  \{ f_{j;L_{j',j,J_{j'},t},L_{j',j,J_{j'},t}}\mid 1 \leq t \leq \mathsf M_{i+1;j',j}  \}$ for $1 \leq j' \leq \mathsf K_{i}$; hence, for some $I_j \in \{L_{j',j,J_{j'},t} \mid 1 \leq j' \leq \mathsf K_{i},\, 1 \leq t \leq \mathsf M_{i+1;j',j}\}$, we have $W_{i+1;j',j} = \{f_{j;I_j,L_{j',j,J_{j'},t}} \mid 1 \leq t \leq \mathsf M_{i+1;j',j}\}$ and $U_{i+1,j} = \{f_{j;I_j,k} \mid 1 \leq k \leq \mathsf N_{i+1,j}\}$. But Conditions \ref{prop1} and \ref{prop2} imply that the set of numbers $\{L_{j',j,k,t}\mid 1 \leq j' \leq \mathsf K_i,\, 1 \leq k \leq \mathsf N_{i,j'},\, 1 \leq t \leq \mathsf M_{i+1;j',j}\}$ exhausts the integers in the interval $[1,\mathsf N_{i+1,j}]$. Hence, we can rewrite $U_{i+1,j}$ as 
\begin{align*}
	U_{i+1,j} = \{f_{j;I_j,L_{j',j,k,t}}\mid 1 \leq j' \leq \mathsf K_i,\, 1 \leq k \leq \mathsf N_{i,j'},\, 1 \leq t \leq \mathsf M_{i+1;j',j}\}.
\end{align*}
Finally, since $f_{j;I_j,L_{j',j,J_{j'},t}}e_{j';J_{j'},k} = f_{j; I_j,L_{j',j,k,t}}$ for $1 \leq j' \leq \mathsf K_i$, $1 \leq k \leq \mathsf N_{i,j'}$, and $1 \leq t \leq \mathsf M_{i+1;j',j}$, the result follows. 
\end{proof}

	Let $\{d_1,d_2,\dots\}$ be a subset of self-adjoint generators for $D$. Then, for each $i \in \mathbb N$, by Lemma \ref{L:preMain2}, there exists a subset $G_i = \{g_{i,j}\}_{1\leq j \leq \mathsf K_i} \subseteq B$ such that 
\begin{enumerate}[label=(G\arabic*),ref=G\arabic*]
\item $g_{i,j} \in w_{i;1,j,1}Bw_{i;1,j,1}$, \label{Con:G1}
\item $0 \not \in \sigma(g_{i,j})$, \label{Con:G2}
\item $d_i \in \C(U_{i,1},\dotsc,U_{i,\mathsf K_i},G_i)$; \label{Con:G3}
\end{enumerate}
moreover, we may assume (using the functional calculus)
\begin{enumerate}[label=(G\arabic*), resume, ref=G\arabic*]
\item $\sigma(g_{i,j}) \cap \sigma(g_{i',j'}) = \text{\O}$ for $i \not = i'$ or $j \not = j'$,\label{Con:G4}
\item $\|g_{i,j}\|\leq 2^{-i-j-2}$. \label{Con:G5}
\end{enumerate}
Also, let $\Lambda = \{\lambda_{i;j',j,k} \mid i \in \mathbb N,\, 1 \leq j' \leq \mathsf K_{i-1},\, 1 \leq j \leq \mathsf K_i,\, 1 \leq k \leq \mathsf M_{i;j',j}\}$ be a set of mutually different positive real numbers such that $\Lambda \cap (\bigcup_{i\in \mathbb N} \bigcup_{1\leq j \leq \mathsf K_i}\sigma (g_{i,j})) = \text{\O}$ and 
\begin{align*}
\sum_{j'=1}^{\mathsf K_{i-1}}\sum_{j=1}^{\mathsf K_i}\sum_{k=1}^{\mathsf M_{i;j',j}} \lambda_{i;j',j,k} \leq 2^{-i-5}, \quad \forall i \in \mathbb N. 
\end{align*}

	With the necessary ingredients now defined, we claim that a generator for $B(A,D)$ is given by $\mathfrak G = \sum_{i\in \mathbb N} \mathfrak G_i$, where 
\begin{align*}
	\mathfrak G_i = \begin{cases}\sum\limits_{j=1}^{\mathsf K_i} \bigg ( g_{i,j} + \sum\limits_{k=2}^{\mathsf M_{i;1,j}-1} \lambda_{i;1,j,k} q_{i;1,j,k} + \sum\limits_{k=2}^{\mathsf M_{i;1,j}} \lambda_{i;1,j,k} w_{i;1,j,k}\bigg ) ,& \mathsf K_{i-1} = 1\\[10pt] 
	\sum\limits_{j=1}^{\mathsf K_i} \bigg ( g_{i,j} + \sum\limits_{k=2}^{\mathsf M_{i;1,j}} \lambda_{i;1,j,k} q_{i;1,j,k} + \sum\limits_{k=2}^{\mathsf M_{i;1,j}} \lambda_{i;1,j,k} w_{i;1,j,k}\\
	 \quad + \sum\limits_{j'=2}^{\mathsf K_{i-1}-1}  \sum\limits_{k=1}^{\mathsf M_{i;j',j}} \big (\lambda_{i;j',j,k} q_{i;j',j,k} + \lambda_{i;j',j,k} w_{i;j',j,k} \big ) \\
	\qquad + \sum\limits_{k=1}^{\mathsf M_{i;\mathsf K_{i-1},j}-1} \lambda_{i;\mathsf K_{i-1},j,k} q_{i;\mathsf K_{i-1},j,k} + \sum\limits_{k=1}^{\mathsf M_{i;\mathsf K_{i-1},j}} \lambda_{i;\mathsf K_{i-1},j,k} w_{i;\mathsf K_{i-1},j,k} \bigg ),& \mathsf K_{i-1} \not = 1 \end{cases}. 
\end{align*}
It is plain to see 
\begin{align}\label{E:inequality}
	\|\mathfrak G_i\| < 2^{-i-2} + 8 \sum_{j'=1}^{\mathsf K_{i-1}}\sum_{j=1}^{\mathsf K_i}\sum_{k=1}^{\mathsf M_{i;j',j}} \lambda_{i;j',j,k} \leq 2^{-i-1}
\end{align}
so that $\mathfrak G$ is well-defined. 

	Before proving that $\mathfrak G$ generates $B(A,D)$, we perform some straightforward calculations which will be needed in the proof. In what follows, we drop the indices on the elements of $\Lambda$ to make our equations more readable. It is immaterial which specific member of $\Lambda$ is attached to which partial isometry; what is important is that for distinct partial isometries, there are distinct members of $\Lambda$ attached to each. Nevertheless, for the sake of rigor, we make the following rule. Whenever an expression of the form $a\lambda u_{i;j',j,k}b$ appears below, for $a,b \in B$ and $u_{i;j',j,k} \in Q_{i;j',j} \cup W_{i;j',j}$ (for any $i \in \mathbb N$, $1 \leq j' \leq \mathsf K_{i-1}$, and $1 \leq j \leq \mathsf K_i$), it is to be interpreted as $a \lambda_{i;j',j,k}u_{i;j',j,k} b$ for $\lambda_{i;j',j,k} \in \Lambda$. 
	
	Define 
\begin{align*}
	S := \bigcup_{i\in \mathbb N} \bigcup_{j'=1}^{\mathsf K_{i-1}} \bigcup_{j=1}^{\mathsf K_i} Q_{i;j',j}, 
\end{align*}
and let $q_{i;j',j,k},q_{r;s',s,t}\in S$. We calculate the product $q_{i;j',j,k}q_{r;s',s,t}$ for $i+1 <r$, $i+1 = r$, and $i= r$. Indeed, for $i+1 <r$,  
\begin{align}\label{E:prod1}
	q_{i;j',j,k}q_{r;s',s,t} = \begin{cases}
	q_{r;s',s,t},& j'=\mathsf K_{i-1},\, j=\mathsf K_i,\, k=\mathsf M_{i;\mathsf K_{i-1},\mathsf K_i}\\
		0,& \text{otherwise}
	\end{cases}, 
\end{align}
and 
\begin{gather}
	q_{i;j',j,k}q_{i+1;s',s,t} = \begin{cases}
		q_{i+1;s',s,t},& j'=\mathsf K_{i-1},\, j=s',\, k=\mathsf M_{i;\mathsf K_{i-1},j}\\
		0,& \text{otherwise}
		\end{cases},\label{E:prod2}\\
		q_{i;j',j,k}q_{i;s',s,t} = \begin{cases}
		q_{i;s',s,t},& j'=s',\, j=s,\, k=t\\
		0,& \text{otherwise}\label{E:prod3}
	\end{cases}. 
\end{gather}
Taking adjoints and relabeling indices yields the products for $i-1=r$ and $i-1>r$. 

	Now consider the subset of $S$ given by $R = S \setminus \{q_{i;\mathsf K_{i-1},j,\mathsf M_{i;\mathsf K_{i-1},j}}\mid i \in \mathbb N,\, 1 \leq j \leq \mathsf K_i\}$, and let $q_{r;s',s,t} \in R$. Then 
\begin{align*}
	q_{r;s',s,t} \mathfrak G_i = \begin{cases}\sum\limits_{j=1}^{\mathsf K_i} \bigg ( q_{r;s',s,t} q_{i;1,j,1} g_{i,j}\\ \quad + \sum\limits_{k=2}^{\mathsf M_{i;1,j}-1} q_{r;s',s,t} \lambda q_{i;1,j,k}     + \sum\limits_{k=2}^{\mathsf M_{i;1,j}} q_{r;s',s,t}q_{i;1,j,1}\lambda w_{i;1,j,k}\bigg ) ,& \mathsf K_{i-1} = 1\\[10pt] 
	\sum\limits_{j=1}^{\mathsf K_i} \bigg ( q_{r;s',s,t} q_{i;1,j,1}g_{i,j}   + \sum\limits_{k=2}^{\mathsf M_{i;1,j}} q_{r;s',s,t}\lambda q_{i;1,j,k}  +  \sum\limits_{k=2}^{\mathsf M_{i;1,j}} q_{r;s',s,t} q_{i;1,j,1}\lambda w_{i;1,j,k}\\
	 \quad + \sum\limits_{j'=2}^{\mathsf K_{i-1}-1}  \sum\limits_{k=1}^{\mathsf M_{i;j',j}} \big (q_{r;s',s,t} \lambda q_{i;j',j,k} + q_{r;s',s,t}q_{i;1,j,1}\lambda w_{i;j',j,k} \big ) \\
	\qquad + \sum\limits_{k=1}^{\mathsf M_{i;\mathsf K_{i-1},j}-1} q_{r;s',s,t}\lambda q_{i;\mathsf K_{i-1},j,k} + \sum\limits_{k=1}^{\mathsf M_{i;\mathsf K_{i-1},j}} q_{r;s',s,t}q_{i;1,j,1} \lambda w_{i;\mathsf K_{i-1},j,k} \bigg ),& \mathsf K_{i-1} \not = 1 \end{cases}
\end{align*}
so that, using Equations \eqref{E:prod1}--\eqref{E:prod3}, we find $q_{r;s',s,t}\mathfrak G_i = 0$ when $i \not = r$ and 
\begin{align}\label{E:qG}
	q_{r;s',s,t}\mathfrak G &= q_{r;s',s,t} \mathfrak G_r \\ &=  \begin{cases}
		g_{r,s} + \sum\limits_{k=2}^{\mathsf M_{r;1,s}} \lambda w_{r;1,s,k} \\ \quad + \sum\limits_{j'=2}^{\mathsf K_{r-1}-1} \sum\limits_{k=1}^{\mathsf M_{r;j',s}} \lambda w_{r;j',s,k} + (1-\delta_{1,\mathsf K_{r-1}}) \sum\limits_{k=1}^{\mathsf M_{r;\mathsf K_{r-1},s}} \lambda w_{r;\mathsf K_{r-1},s,k},& s'=1,\, t=1\\
		\lambda q_{r;s',s,t},& \text{otherwise}
	\end{cases}. \notag
\end{align}
It follows that 
\begin{align}\label{E:qGq}
	q_{r;s',s,t}\mathfrak G q_{r;s',s,t} &=  \begin{cases}
		g_{r,s}q_{i;1,s,1}q_{r;s',s,t} + \sum\limits_{k=2}^{\mathsf M_{r;1,s}} \lambda w_{r;1,s,k}q_{r;1,s,k}q_{r;s',s,t} \\ \quad + \sum\limits_{j'=2}^{\mathsf K_{r-1}-1} \sum\limits_{k=1}^{\mathsf M_{r;j',s}} \lambda w_{r;j',s,k} q_{r;j',s,k}q_{r;s',s,t}\\ \qquad+  (1-\delta_{1,\mathsf K_{r-1}}) \sum\limits_{k=1}^{\mathsf M_{r;\mathsf K_{r-1},s}} \lambda w_{r;\mathsf K_{r-1},s,k}q_{r;\mathsf K_{r-1},s,k}q_{r;s',s,t},& s'=1,\, t=1\\
		\lambda q_{r;s',s,t}q_{r;s',s,t},& \text{otherwise}
	\end{cases}
	\\ &= \begin{cases}
		g_{r,s},& s'=1,\, t=1\\
		\lambda q_{r;s',s,t},& \text{otherwise}
	\end{cases}. \notag
\end{align}
Moreover, 
\begin{align*}
	\mathfrak G_i q_{r;s',s,t}= \begin{cases}\sum\limits_{j=1}^{\mathsf K_i} \bigg ( g_{i,j}q_{i;1,j,1}q_{r;s',s,t} \\ \quad + \sum\limits_{k=2}^{\mathsf M_{i;1,j}-1} \lambda q_{i;1,j,k}q_{r;s',s,t} + \sum\limits_{k=2}^{\mathsf M_{i;1,j}} \lambda w_{i;1,j,k}q_{i;1,j,k}q_{r;s',s,t}\bigg ) ,& \mathsf K_{i-1} = 1\\[10pt] 
	\sum\limits_{j=1}^{\mathsf K_i} \bigg ( g_{i,j}q_{i;1,j,1}q_{r;s',s,t} + \sum\limits_{k=2}^{\mathsf M_{i;1,j}} \lambda q_{i;1,j,k}q_{r;s',s,t} + \sum\limits_{k=2}^{\mathsf M_{i;1,j}} \lambda w_{i;1,j,k}q_{i;1,j,k} q_{r;s',s,t}\\
	 \quad + \sum\limits_{j'=2}^{\mathsf K_{i-1}-1}  \sum\limits_{k=1}^{\mathsf M_{i;j',j}} \big (\lambda q_{i;j',j,k} q_{r;s',s,t} + \lambda w_{i;j',j,k} q_{i;j',j,k}q_{r;s',s,t}\big ) \\
	\qquad + \sum\limits_{k=1}^{\mathsf M_{i;\mathsf K_{i-1},j}-1} \lambda q_{i;\mathsf K_{i-1},j,k} q_{r;s',s,t} + \sum\limits_{k=1}^{\mathsf M_{i;\mathsf K_{i-1},j}} \lambda w_{i;\mathsf K_{i-1},j,k} q_{i;\mathsf K_{i-1},j,k} q_{r;s',s,t}\bigg ),& \mathsf K_{i-1} \not = 1 \end{cases}, 
\end{align*}
so that, using Equations \eqref{E:prod1}--\eqref{E:prod3} again, we find $\mathfrak G_i q_{r;s',s,t} = 0$ when $i >r$, $\mathfrak G_i q_{r;s',s,t} = \lambda w_{i;\mathsf K_{i-1},\mathsf K_i,\mathsf M_{i;\mathsf K_{i-1},\mathsf K_i}} q_{r;s',s,t}$ when $i+1 <r$, and  
\begin{gather*}
	\mathfrak G_{r-1} q_{r;s',s,t} = \lambda w_{r-1;\mathsf K_{r-2},s',\mathsf M_{r-1;\mathsf K_{r-2},s'}} q_{r;s',s,t},\\
	\mathfrak G_{r} q_{r;s',s,t} = \begin{cases}
		g_{r,s},& s'=1,\, t=1\\
		\lambda q_{r;s',s,t} + \lambda w_{r;s',s,t},& \text{otherwise}
	\end{cases}. 
\end{gather*}
Hence, 
\begin{align}\label{E:Gq}
	\mathfrak Gq_{r;s',s,t} &= \sum_{i=1}^r \mathfrak G_i q_{r;s',s,t}\\
	&= \sum_{i=1}^{r-2} \lambda w_{i;\mathsf K_{i-1},\mathsf K_i,\mathsf M_{i;\mathsf K_{i-1},\mathsf K_i}} q_{r;s',s,t} + \lambda w_{r-1;\mathsf K_{r-2},s',\mathsf M_{r-1;\mathsf K_{r-2},s'}} q_{r;s',s,t} + \mathfrak G_{r} q_{r;s',s,t}\notag\\
	&= \begin{cases}
	\sum\limits_{i=1}^{r-2} \lambda w_{i;\mathsf K_{i-1},\mathsf K_i,\mathsf M_{i;\mathsf K_{i-1},\mathsf K_i}} q_{r;1,s,1} \\ \quad + (1-\delta_{r,1})\lambda w_{r-1;\mathsf K_{r-2},1,\mathsf M_{r-1;\mathsf K_{r-2},1}} q_{r;1,s,1} + g_{r,s},& s'=1, \, t=1\\
	\sum\limits_{i=1}^{r-2} \lambda w_{i;\mathsf K_{i-1},\mathsf K_i,\mathsf M_{i;\mathsf K_{i-1},\mathsf K_i}} q_{r;s',s,t} \\ \quad + (1-\delta_{r,1})\lambda w_{r-1;\mathsf K_{r-2},s',\mathsf M_{r-1;\mathsf K_{r-2},s'}} q_{r;s',s,t} \\ \qquad + \lambda q_{r;s',s,t} + \lambda w_{r;s',s,t},& \text{otherwise}
	\end{cases}.\notag
\end{align}

\begin{thm}\label{T:main}
	A C*-algebra with an AF-action is singly generated. 
\end{thm}

\begin{proof}
	We will prove that $\C(\mathfrak G) = B(A,D)$. For this, our goal is to show that $R \subseteq \C(\mathfrak G)$. Once this is done, we will be able to use the elements of $R$ to extract the finite-dimensional algebras $\bigoplus_{1\leq j\leq \mathsf K_i} \M_{\mathsf N_{i,j}}$ (and hence the AF algebra $A$) from $\C(\mathfrak G)$ along with the self-adjoint generators $d_1,d_2,\dots$ of $D$. Since $B = \C(A,D)$, the result will then follow. 

	Let $\preceq$ denote the lexicographic order on $R$; to be precise, $q_{i;j',j,k} \preceq q_{r;s',s,t}$ if $i < r$ or if $i=r$ and $j' < s'$ or if $i=r$, $j'=s'$, and $j < s$ or if $i=r$, $j'=s'$, $j=s$, and $k \leq t$. Let $p_1 =q_{1;1,1,1}$, and for every $i \in \mathbb N$, define $p_{i+1}\in R$ such that $p_{i+1} \preceq q$ for every $q \in R\setminus \{p_1,\dotsc,p_i\}$ (roughly speaking, $p_{i+1}$ is the smallest element in $R$ greater than $p_i$). To show that $R \subseteq \C(\mathfrak G)$, it is sufficient to show that $\mathfrak G$ and the sequence $(p_i)_{i\in \mathbb N}$ of nonzero mutually orthogonal projections satisfy the hypotheses of Lemma \ref{L:upT}. That $\mathfrak G$ and $(p_i)_{i\in \mathbb N}$ satisfy Conditions \ref{Con:upTspec} and \ref{Con:upTspec2} of Lemma \ref{L:upT} is clear from the spectral properties of the members of $\bigcup_{r\in \mathbb N} G_r$ (in particular, from Conditions \ref{Con:G2} and \ref{Con:G4}), the definition of $\Lambda$, and Equation \eqref{E:qGq}. 
	
	We now show that Condition \ref{Con:upT} of Lemma \ref{L:upT} holds; defining $P_n := \sum_{1\leq i \leq n} p_i$, we wish to show $(1-P_n) \mathfrak G P_n = 0$ for every $n \in \mathbb N$. Appealing to Equations \eqref{E:qGq} and  \eqref{E:Gq}, we have $(1-P_1)\mathfrak G P_1 = \mathfrak Gq_{1;1,1,1} - q_{1;1,1,1}\mathfrak G q_{1;1,1,1} = g_{1,1} - g_{1,1} = 0$ so that the desired equality is true for the case $n = 1$. Fix $n \in \mathbb N$, and suppose $(1-P_n)\mathfrak G P_n = 0$. Notice 
\begin{align*}
	(1-P_{n+1}) \mathfrak G P_{n+1} &= \big (1-(P_n + p_{n+1}) \big ) \mathfrak G(P_n + p_{n+1})\\ &= (1-P_n)\mathfrak G P_n + \mathfrak G p_{n+1} - P_n \mathfrak G p_{n+1} - p_{n+1} \mathfrak G P_n - p_{n+1} \mathfrak G p_{n+1}; 
\end{align*}	
thus, to ensure $(1-P_{n+1}) \mathfrak G P_{n+1} = 0$, we need 
\begin{align}\label{E:pGp}
	\mathfrak G p_{n+1} - P_n \mathfrak G p_{n+1} - p_{n+1} \mathfrak G P_n - p_{n+1} \mathfrak G p_{n+1} = 0. 
\end{align}

	To that end, assume $p_{n+1} = q_{r;s',s,t}$, and notice from the definition of $P_n$ that 
\begin{align}\label{E:Pp}
	P_n p_i = p_i P_n = \begin{cases}
		p_i,& i \leq n\\
		0,& i > n
	\end{cases}. 
\end{align}
Hence, appealing to Equation \eqref{E:Gq}, 
\begin{align}\label{E:PGq}
	P_n\mathfrak Gp_{n+1} &= \begin{cases}
	\sum\limits_{i=1}^{r-2} P_n q_{i;1,\mathsf K_i,1}\lambda w_{i;\mathsf K_{i-1},\mathsf K_i,\mathsf M_{i;\mathsf K_{i-1},\mathsf K_i}} q_{r;1,s,1} \\ \quad + (1-\delta_{r,1})P_n q_{r-1;1,1,1}\lambda w_{r-1;\mathsf K_{r-2},1,\mathsf M_{r-1;\mathsf K_{r-2},1}} q_{r;1,s,1}\\ \qquad + P_n q_{r;1,s,1}g_{r,s},& s'=1, \, t=1\\
	\sum\limits_{i=1}^{r-2} P_n q_{i;1,\mathsf K_i,1} \lambda w_{i;\mathsf K_{i-1},\mathsf K_i,\mathsf M_{i;\mathsf K_{i-1},\mathsf K_i}} q_{r;s',s,t} \\ \quad + (1-\delta_{r,1})P_n q_{r-1;1,s',1}\lambda w_{r-1;\mathsf K_{r-2},s',\mathsf M_{r-1;\mathsf K_{r-2},s'}} q_{r;s',s,t} \\ \qquad + P_n\lambda q_{r;s',s,t} + P_n q_{r;1,s,1}\lambda w_{r;s',s,t},& \text{otherwise}
	\end{cases} \\
	&= \begin{cases}
		\sum\limits_{i=1}^{r-2} \lambda w_{i;\mathsf K_{i-1},\mathsf K_i,\mathsf M_{i;\mathsf K_{i-1},\mathsf K_i}} q_{r;1,s,1} \\ \quad + (1-\delta_{r,1})\lambda w_{r-1;\mathsf K_{r-2},1,\mathsf M_{r-1;\mathsf K_{r-2},1}} q_{r;1,s,1},& s'=1,\, t=1\\
		\sum\limits_{i=1}^{r-2} \lambda w_{i;\mathsf K_{i-1},\mathsf K_i,\mathsf M_{i;\mathsf K_{i-1},\mathsf K_i}} q_{r;s',s,t} \\ \quad + (1-\delta_{r,1})\lambda w_{r-1;\mathsf K_{r-2},s',\mathsf M_{r-1;\mathsf K_{r-2},s'}} q_{r;s',s,t} + \lambda w_{r;s',s,t},& \text{otherwise}
	\end{cases}, \notag
\end{align}
and appealing to Equation \eqref{E:qG},
\begin{align}\label{E:qGP}
	p_{n+1}\mathfrak G P_n &=  \begin{cases}
		g_{r,s}q_{r;1,s,1}P_n + \sum\limits_{k=2}^{\mathsf M_{r;1,s}} \lambda w_{r;1,s,k}q_{r;1,s,k}P_n \\ \quad + \sum\limits_{j'=2}^{\mathsf K_{r-1}-1} \sum\limits_{k=1}^{\mathsf M_{r;j',s}} \lambda w_{r;j',s,k} q_{r;j',s,k}P_n \\ \qquad + (1-\delta_{1,\mathsf K_{r-1}}) \sum\limits_{k=1}^{\mathsf M_{r;\mathsf K_{r-1},s}} \lambda w_{r;\mathsf K_{r-1},s,k}q_{r;\mathsf K_{r-1},s,k}P_n,& s'=1,\, t=1\\
		\lambda q_{r;s',s,t}P_n,& \text{otherwise}
	\end{cases} \\ &= 0. \notag
\end{align}
Thus, we see from Equations \eqref{E:Gq}, \eqref{E:PGq}, \eqref{E:qGP}, and \eqref{E:qGq} that in fact Equation \eqref{E:pGp} holds; that is, $\mathfrak G$ and $(p_i)_{i\in \mathbb N}$ satisfy Condition \ref{Con:upT} of Lemma \ref{L:upT}. 

	Finally, to see Condition \ref{Con:trace} of Lemma \ref{L:upT} holds for $\mathfrak G$ and $(p_i)_{i\in \mathbb N}$, we wish to show 
\begin{align*}
	\lim_{n\to\infty} \| (1-P_n)\mathfrak G (1-P_n)\| = \lim_{n\to\infty}\| (1-P_n)\mathfrak G - (1-P_n)\mathfrak G P_n\| = \lim_{n\to\infty}\|\mathfrak G - P_n\mathfrak G\| = 0, 
\end{align*}
where the second equality follows from what we just proved in the previous paragraph. Fix $n \in \mathbb N$, and assume $p_{n+1} = q_{r;s',s,t}$ again. Notice 
\begin{align*}
	P_n \mathfrak G_i = \begin{cases}\sum\limits_{j=1}^{\mathsf K_i} \bigg ( P_n q_{i;1,j,1}g_{i,j} + \sum\limits_{k=2}^{\mathsf M_{i;1,j}-1} P_n \lambda q_{i;1,j,k} + \sum\limits_{k=2}^{\mathsf M_{i;1,j}} P_n  q_{i;1,j,1} \lambda w_{i;1,j,k}\bigg ) ,& \mathsf K_{i-1} = 1\\[10pt] 
	\sum\limits_{j=1}^{\mathsf K_i} \bigg ( P_n q_{i;1,j,1}g_{i,j} + \sum\limits_{k=2}^{\mathsf M_{i;1,j}} P_n\lambda q_{i;1,j,k} + \sum\limits_{k=2}^{\mathsf M_{i;1,j}} P_n q_{i;1,j,1}\lambda w_{i;1,j,k}\\
	 \quad + \sum\limits_{j'=2}^{\mathsf K_{i-1}-1}  \sum\limits_{k=1}^{\mathsf M_{i;j',j}} \big (P_n \lambda q_{i;j',j,k} + P_n q_{i;1,j,1} \lambda w_{i;j',j,k} \big ) \\
	\qquad + \sum\limits_{k=1}^{\mathsf M_{i;\mathsf K_{i-1},j}-1} P_n \lambda q_{i;\mathsf K_{i-1},j,k} + \sum\limits_{k=1}^{\mathsf M_{i;\mathsf K_{i-1},j}} P_n q_{i;1,j,1} \lambda w_{i;\mathsf K_{i-1},j,k} \bigg ),& \mathsf K_{i-1} \not = 1 \end{cases}. 
\end{align*}
That is, $P_n \mathfrak G_i$ is a sum of terms of the form $P_n qb$ for $q \in ( \bigcup_{1 \leq j' \leq \mathsf K_{i-1}} \bigcup_{1 \leq j \leq \mathsf K_i} Q_{i;j',j}) \cap R$ and $b \in B$; in particular, by Equation \eqref{E:Pp}, $P_n q b=0$ for $p_{n+1} \preceq q$ and $P_n q b = qb$ otherwise. It follows that 
\begin{align*}
	P_n \mathfrak G_i = \begin{cases}
		\mathfrak G_i,& 1 \leq i <r\\
		0,& i >r
	\end{cases}, 
\end{align*}
and subsequently, that 
\begin{align*}
	\|\mathfrak G - P_n \mathfrak G\| = \Big \|\sum_{i=r+1}^{\infty} \mathfrak G_{i} + \mathfrak G_r - P_n \mathfrak G_r \Big \|<  \sum_{i=r+1}^\infty 2^{-i-1} + \|\mathfrak G_r - P_n \mathfrak G_r\| < \sum_{i=r+1}^\infty 2^{-i-1} + 2^{-r-1}.
\end{align*}
Noticing that as $n$ goes to infinity so does $r$, Condition \ref{Con:trace} of Lemma \ref{L:upT} follows. We conclude that $R \subseteq \C(\mathfrak G)$. 

Now, notice $W_{r;s',s} \subseteq \C(\mathfrak G)$ for every $r \in \mathbb N$, $1 \leq s' \leq \mathsf K_{r-1}$ and $1 \leq s \leq \mathsf K_r$. Indeed, for any $r \in \mathbb N$ and $1 \leq s \leq \mathsf K_r$, $w_{r;1,s,1} \in \C(\mathfrak G)$ since $w_{r;1,s,1} = q_{r;1,s,1} \in R$; moreover, from Equation \eqref{E:qG} 
\begin{gather*}
	  \frac{1}{\lambda_{r;1,s,k}}(q_{r;1,s,1}\mathfrak G - g_{r,s})q_{r;1,s,k} = w_{r;1,s,k} \in \C(\mathfrak G),\quad 2\leq k \leq \mathsf M_{r;1,s},\\
	  \frac{1}{\lambda_{r;j',s,k}}(q_{r;1,s,1}\mathfrak G - g_{r,s})q_{r;j',s,k} = w_{r;j',s,k} \in \C(\mathfrak G),\quad 1 < j' < \mathsf K_{r-1},\, 1 \leq k \leq \mathsf M_{r;j',s},\\
	  \frac{1}{\lambda_{r;\mathsf K_{r-1},s,k}}(q_{r;1,s,1}\mathfrak G - g_{r,s})q_{r;\mathsf K_{r-1},s,k} = w_{r;\mathsf K_{r-1},s,k} \in \C(\mathfrak G),\quad 1 \leq k < \mathsf M_{r;\mathsf K_{r-1},s}; 
\end{gather*}
also, 
\begin{multline*}
	 \frac{1}{\lambda_{r;\mathsf K_{r-1},s,\mathsf M_{r;\mathsf K_{r-1},s}}}\big (q_{r;1,s,1}\mathfrak G-g_{r,s}  - \sum\limits_{k=2}^{\mathsf M_{r;1,s}} \lambda w_{r;1,s,k} - \sum\limits_{j'=2}^{\mathsf K_{r-1}-1} \sum\limits_{k=1}^{\mathsf M_{r;j',s}} \lambda w_{r;j',s,k}\\ - (1-\delta_{1,\mathsf K_{r-1}}) \sum\limits_{k=1}^{\mathsf M_{r;\mathsf K_{r-1},s}-1} \lambda w_{r;\mathsf K_{r-1},s,k}\big ) = w_{r;\mathsf K_{r-1},s,\mathsf M_{r;\mathsf K_{r-1},s}} \in \C(\mathfrak G). 
\end{multline*}

Since $U_{1,j} = W_{1;1,j}$ for each $1\leq j \leq \mathsf K_1$, we see $\bigoplus_{1\leq j\leq \mathsf K_1} \M_{\mathsf N_{1,j}} \subseteq \C(\mathfrak G)$ (see the discussion following Lemma \ref{L:preMain2}); hence, by Lemma \ref{L:partialIso}, $\bigoplus_{1\leq j\leq \mathsf K_i} \M_{\mathsf N_{i,j}} \subseteq \C(\mathfrak G)$ for each $i \in \mathbb N$, and we see $A \subseteq \C(\mathfrak G)$. Furthermore, it is  clear from Equation \eqref{E:qGq} that $G_i \subseteq \C(\mathfrak G)$ for each $i \in \mathbb N$; but $d_i \in \C(\bigoplus_{1\leq j\leq \mathsf K_i} \M_{\mathsf N_{i,j}}, G_i)$ by Condition \ref{Con:G3} so that $\{d_1,d_2,\dots\}$ and hence $D$ is contained in $\C(\mathfrak G)$. 
\end{proof}

	The following corollaries now follow from the discussion at the end of Section \ref{S:prelim}. 

\begin{cor}\label{C:main}
	A simple AH algebra with diagonal maps is singly generated. 
\end{cor}

\begin{cor}\label{C:main2}
	A Villadsen algebra is singly generated. 
\end{cor}

\begin{cor}
	Let $B=B(A,D)$ have an AF-action, and let $C$ be a separable unital C*-algebra. Then $B\otimes C$ is singly generated. In particular, if $B$ is a Villadsen algebra, then $B\otimes C$ is singly generated. 
\end{cor}

\end{document}